\documentclass[a4paper,12pt]{amsart}
\usepackage{amsmath,amssymb,amsfonts,enumerate}
\usepackage{a4}

\newcommand\Beta{\mathsf{B}}
\newcommand\G{\Gamma}

\newtheorem{theorem}{Theorem}
\newtheorem{proposition}[theorem]{Proposition}
\theoremstyle{definition}

\begin{document}
\vspace*{2cm}

\begin{center}
{\Large Directionally $2$-Signed and Bidirected Graphs}
\\[20pt]

{\large E.\ Sampathkumar and M.\ A.\ Sriraj}
\\[10pt]

{Department of Mathematics, University of Mysore, Mysore, India}
\\[10pt]

{\large Thomas Zaslavsky}
\\[10pt]

{Department of Mathematical Sciences, Binghamton University (SUNY), Binghamton, NY 13902-6000, U.S.A.}
\\[10pt]

{\today}
\\[10pt]

\small{e-mail: esampathkumar@gmail.com, srinivasa\_sriraj@yahoo.co.in, zaslav\@math.binghamton.edu}
\\[20pt]
\end{center}

{\small

\begin{quote}
\emph{Abstract.}
An edge $uv$ in a graph $\G$ is \emph{directionally $2$-signed} (or, \emph{$(2,d)$-signed}) by an ordered pair $(a,b)$, $a,b \in \{+,-\}$, if the label $l(uv) = (a,b)$ from $u$ to $v$, and $l(vu) = (b,a)$ from $v$ to $u$.  Directionally 2-signed graphs are equivalent to bidirected graphs, where each end of an edge has a sign.  A bidirected graph implies a signed graph, where each edge has a sign.  We extend a theorem of Sriraj and Sampathkumar by proving that the signed graph is antibalanced (all even cycles and only even cycles have positive edge sign product) if, and only if, in the bidirected graph, after suitable reorientation of edges every vertex is a source or a sink.
\bigskip

\textbf{Keywords}:  Signed graph, Directionally multisigned graph, $(n,d)$-Sigraph, Bidirected graph, Uniform directional labeling, Antibalance, Skew gains
\\

\textbf{Mathematics Subject Classification 2010}: 05C22

\end{quote}
}
\bigskip

\emph{Signed graphs}, in which the edges of a graph are labelled positive or negative, have developed many applications and a flourishing literature (see \cite{BSG}) since their first introduction by Harary in 1953 \cite{NB}.  Their natural extension to \emph{multisigned graphs}, in which each edge gets an $n$-tuple of signs---that is, the sign group is replaced by a direct product of sign groups---has received slight attention, but the further extension to \emph{gain graphs} (also known as \emph{voltage graphs}), which have edge labels from an arbitrary group such that reversing the edge orientation inverts the label, have been well studied \cite{BSG}.  Note that in a multiple sign group every element is its own inverse, so the question of edge reversal does not arise with multisigned graphs.

Recently, Sampathkumar, Siva Kota Reddy, and Subramanya \cite{Dn, Dn2} introduced a new idea: \emph{directionally multisigned} (or \emph{$(n,d)$-signed}) graphs, assigning to each edge of a graph an $n$-tuple of signs with the rule that reversing the direction of the edge reverses the direction of the $n$-tuple; in other words, if $l(uv) = (a_1,a_2,\ldots,a_n)$ in the direction from $u$ to $v$, then $l(vu) = (a_n,\ldots,a_2,a_1)$ from $v$ to $u$.  Directionally multisigned graphs are not gain graphs, because the effect of edge reorientation on the label is not inversion; but they are a tantalizingly simple example of the also recent idea of \emph{skew gain graphs} \cite{HageSkew, HHSkew}, in which edge reversal corresponds to applying an involutory group antiautomorphism that need not be inversion.  

Signed graphs are the special case $n=1$, where directionality is trivial.  Directionally 2-signed graphs (or \emph{$(2,d)$-signed} graphs) are also special, in a less obvious way.  
A \emph{bidirected graph} $\Beta = (\G,\beta)$ ($\Beta$ is capital Beta) is a graph $\G = (V,E)$ in which each end $(e,u)$ of an edge $e=uv$ has a sign $\beta(e,u) \in \{+,-\}$.  $\G$ is the \emph{underlying graph} and $\beta$ is the \emph{bidirection}.  (The $+$ sign denotes an arrow on the $u$-end of $e$ pointed into the vertex $u$; a $-$ sign denotes an arrow directed out of $u$.  Thus, in a bidirected graph each end of an edge has an independent direction.  Bidirected graphs were defined by Edmonds; cf.\ \cite{EJ}.)

Our graphs may have loops and multiple edges.

\begin{theorem}\label{T:di2bi}
Directionally $2$-signed graphs are equivalent to bidirected graphs.
\end{theorem}

\begin{proof}
The equivalence for an edge $e=uv$ is simple; if $l(uv) = (a_1,a_2)$, define $\beta(uv,u) := a_1$.  Since $l(vu) = (a_2,a_1)$, one gets $\beta(uv,v) = a_2$.  Conversely, $\beta$ determines $l(uv) = (\beta(uv,u),\beta(uv,v))$.
\end{proof}

To a bidirected graph $\Beta = (\G,\beta)$ there is associated a signed graph $\Sigma_\Beta = (\G,\sigma_\Beta)$ with edge signature $\sigma_\Beta: E \to \{+,-\}$, in which the sign of an edge is the negative product $\sigma_\Beta(uv) = - \beta(uv,u)\beta(uv,v)$.  One considers $\Beta$ to be an orientation of the signed graph $\Sigma_\Beta$, as explained by Zaslavsky in \cite{OSG}.  For instance, when every edge is positive in $\Sigma_\Beta$, then $\Beta$ is an ordinary digraph.  

Sampathkumar and Sriraj \cite{2D} discovered an elegant property of signed-graph orientations, which we strengthen in Theorem \ref{T:a} after necessary definitions and propositions.  
Note that in \cite{2D} they associate to a directionally $2$-signed graph $(\G,l)$ the negative signed graph, $is(\G,l) := -\Sigma_\Beta := (\G, -\sigma_\Beta)$, called the \emph{induced signed graph} of $(\G,l)$.

A signed graph $\Sigma=(\G,\sigma)$ is \emph{balanced} \cite{NB} if in every cycle the product of the edge signs is positive.  
$\Sigma$ is \emph{antibalanced} \cite{A} if in every even (odd) cycle the product of the edge signs is positive (resp., negative); equivalently, the negated signed graph $-\Sigma = (\G,-\sigma)$ is balanced.  For instance, $\Sigma_\Beta$ is antibalanced if and only if $is(\G,l)$ is balanced.  
The following are the fundamental results about balance, the second being a more advanced form of the first.  Note that in a bipartition of a set, $V = V_1 \cup V_2$, the disjoint subsets may be empty.

\begin{proposition}\label{P:bal} 
A signed graph $\Sigma$ is balanced if and only if either of the following equivalent conditions is satisfied: 
\begin{enumerate}[{\rm (i)}]
\item Its vertex set has a bipartition $V = V_1\cup V_2$  such that every positive edge joins vertices in $V_1$ or in $V_2$, and every negative edge joins a vertex in $V_1$ and a vertex in $V_2$  (Harary \cite{NB}).
\item It is possible to label its vertices with $+$ and $-$ such that the sign of any edge in $\Sigma$ is the product of the signs of its end vertices (Sampathkumar \cite{S}).
\end{enumerate}
 \end{proposition}
 
A vertex $u$ in a bidirected graph is a \emph{source} if $\beta(uv,u) = -$ for every neighbor $v$ and it is a \emph{sink} if $\beta(uv,u) = +$ for every neighbor $v$.  A \emph{uniform vertex} is a source or sink.  We call $\Beta$ \emph{uniformly bidirected} if each vertex is uniform.  Similarly (and originally; see \cite{2D}), a directionally 2-signed graph is called \emph{uniform} if $a_1$, where $l(uv) = (a_1,a_2)$, is independent of $v$ for each $u \in V$.  Clearly, uniformly directionally 2-signed graphs correspond to uniformly bidirected graphs.

\emph{Reorienting an edge} $uv$ means negating $\beta(uv,u)$ and $\beta(uv,v)$.  The corresponding operation on a directionally 2-signed graph is to negate both signs in $l(uv)$.  Note that reorientation of edges in $\Beta$ (equivalently, negation of the directional labeling $l$ on some edges) does not change the associated signed graph $\Sigma_\Beta$.  
We say $\Beta$ is \emph{uniformly bidirected up to reorientation} if it is obtained from a uniform bidirection by reorienting some subset of the edges (possibly, not reorienting any edges).

\begin{theorem}\label{T:a}
A bidirected graph $\Beta$ is uniformly bidirected up to reorientation if and only if $\Sigma_\Beta$ is antibalanced.
\end{theorem}

\begin{proof}
\emph{Sufficiency.}
If $\Sigma_\Beta$ is antibalanced, then by Proposition \ref{P:bal}(ii) there exists a vertex signature $\mu$ such that every negative edge in $\Sigma_\Beta$ joins two vertices with the same sign, and every positive edge joins two vertices with opposite signs; i.e., $\sigma_\Beta(uv) = -\mu(u)\mu(v)$.    Define a bidirected graph $\Beta'$ by $\beta'(uv,u) = \mu(u)$.  By its definition $\Beta'$ is uniformly bidirected.  

The signed graph $\Sigma_{\Beta'}$ determined by $\Beta'$ has the same edge signature as $\Sigma_\Beta$; thus $\Sigma_{\Beta'} = \Sigma_\Beta$.  This implies that $\Beta$ is obtained from the uniformly bidirected graph $\Beta'$ by reorienting the edges at which $\beta$ and $\beta'$ differ.

\emph{Necessity.}
If $\Beta$ is uniformly bidirected up to reorientation, first reorient the edges so $\Beta$ is uniformly bidirected.  This does not change $\Sigma_\Beta$.  Mark each sink with sign $+$ and each source with sign $-$.  By the definition of $\Sigma_\Beta$, the sign of each edge is the negative of the product of the signs of its endpoints.  Thus, by Proposition \ref{P:bal}(ii), $-\Sigma_\Beta$ is balanced and $\Sigma_\Beta$ is antibalanced.
\end{proof}

The theorem of Sampathkumar and Sriraj \cite{2D}, stated in terms of bidirected graphs, is that $-\Sigma_\Beta$ is balanced if $\Beta$ is uniform.  Allowing reorientation enables us to improve the theorem, obtaining equivalence.

The correspondence between bidirected and directionally 2-signed graphs suggests a new way to generalize bidirection, namely, to directional multisigns as skew gains.  
If $n$ is even, $n=2m$, an $n$-tuple sign has the form $$(a_1,a_2,\ldots, a_m,b_m,\ldots,b_2,b_1).$$  Each pair $(a_i,b_i)$ for a fixed $i \in \{1,\ldots,m\}$ is a bidirection of $\G$.  Thus, the directionally $n$-signed graph is an $m$-fold bidirected graph.  The $m$ bidirections are unrelated to each other.
If $n=2m+1$ is odd, an $n$-tuple sign has the form $$(a_1,a_2,\ldots,a_m,c,b_m,\ldots,b_2,b_1).$$  Now we have $m$ bidirections $(a_i,b_i)$ and one sign $c$, for each edge.  
We do not yet know how much this interpretation will help in the theory of directionally $n$-signed graphs.


\end{document}